\documentclass [12pt,reqno]{amsart}
\usepackage{amsmath}
\usepackage{amsthm}
\usepackage{dsfont}
\newtheorem{thm}{Theorem} 
\usepackage{amsfonts, amssymb}
\usepackage[utf8]{inputenc} 
\usepackage[english]{babel}
\usepackage{mathtools}
\usepackage{ mathrsfs }
\usepackage{enumerate}
\usepackage{ amssymb }
\usepackage{graphicx}
\usepackage{cite}

\usepackage{showkeys}
\usepackage{color}
\usepackage[dvipsnames]{xcolor}

\newtheorem{theorem}{Theorem}

\newtheorem{Question}[thm]{Question}
\newtheorem*{theorem*}{Theorem}

\newtheorem{lemma}[thm]{Lemma}
\newtheorem{proposition}[thm]{Proposition}

\newtheorem{corollary}[thm]{Corollary}
\theoremstyle{definition}

\theoremstyle{definition}
\newtheorem{example}[thm]{Example}
\theoremstyle{definition}
\newtheorem{definition}[thm]{Definition}

\newcommand{\SRA}{\operatorname{SRA}}

\newcommand{\id}{\operatorname{id}}

\newcommand{\LRB}{\operatorname{LRB}}

\newcommand{\CAT}{\operatorname{CAT}}
 
\newcommand{\RCD}{\operatorname{RCD}} 
\newcommand{\ATB}{\operatorname{ATB}}
\newcommand{\EVI}{\operatorname{EVI}}
\newcommand{\CD}{\operatorname{CD}}

\numberwithin{equation}{section}

\def\R{{\mathbb R}}

\def\N{{\mathbb N}}

\def\ZZ{{\mathbb Z}}
\def\RR{{\mathbb R}}

\def\wt{\widetilde}
\def\wb{\overline}

\def\BB{{\mathcal B}}

\def\GG{{\Gamma}}

\def\RRR{{\mathscr{R}}}
\def\NNN{{\mathscr{N}}}

\def\M{{\mathcal M}}

\def\<{\langle}
\def\>{\rangle}
\def \ee{{\epsilon}}
\def \dd{{\delta}}

\def \ww{{\omega}}
\def \gg {{\gamma}}
\def \JJ {{\Pi}}
\def \aa {{\alpha}}
\def \bb {{\beta}}
\def \ll {{\lambda}}

\def \TTT {{\Theta}}

\begin{document}

\title{Self-contracted curves in spaces with weak lower curvature bound}

\begin{abstract}
We show that bounded self-contracted curves are rectifiable
in metric spaces with weak lower curvature bound in a sense we introduce in this article.
This class of spaces is wide and includes, for example,
finite-dimensional Alexandrov spaces of curvature bounded below
and Berwald spaces of nonnegative flag curvature.
(To be more precise, our condition is regarded as a strengthened doubling condition
and holds also for a certain class of metric spaces with upper curvature bound.)
We also provide the non-embeddability of large snowflakes into
(balls in) metric spaces in the same class.
We follow the strategy of the last author's previous paper
based on the small rough angle condition,
where spaces with upper curvature bound are considered.
The results in this article show that such a strategy applies
to spaces with lower curvature bound as well.




\end{abstract}

 
\author{Nina Lebedeva}
\address[Nina Lebedeva]{Steklov Institute of Mathematics, Russian Academy of Sciences, 27 Fontanka, 191023 St.Petersburg, Russia}
\email[Nina Lebedeva]{lebed@pdmi.ras.ru}

\author{Shin-ichi Ohta}
\address[Shin-ichi Ohta]{Department of Mathematics, Osaka University, Osaka 560-0043, Japan
\& RIKEN Center for Advanced Intelligence Project (AIP),
1-4-1 Nihonbashi, Tokyo 103-0027, Japan}
\email[Shin-ichi Ohta]{s.ohta@math.sci.osaka-u.ac.jp}

\author{Vladimir Zolotov}
\address[Vladimir Zolotov]{Steklov Institute of Mathematics, Russian Academy of Sciences, 27 Fontanka, 191023 St.Petersburg, Russia 
and Chebyshev Laboratory, St. Petersburg State University, 14th Line V.O., 29B, Saint Petersburg 199178 Russia}
\email[Vladimir Zolotov]{paranuel@mail.ru}

\maketitle

\section{Introduction}

\subsection{Self-contracted curves}

Let $(X,d)$ be a metric space and $I \subset \R$ an interval.
A curve $\gg:I \longrightarrow X$ is said to be \textit{self-contracted} if
\[ d\big( \gg(t_2),\gg(t_3) \big) \le d\big( \gg(t_1),\gg(t_3) \big) \]
holds for every $t_1 \le t_2 \le t_3$
(we remark that $\gg$ is not necessarily continuous).
This notion was introduced in \cite{DLS} and turned out quite useful
for studying the \emph{rectifiability} (finiteness of the length) of $\gg$.
Self-contracted curves arise as gradient curves of quasi-convex functions
(see Definition \ref{quasi-convex}) in ``Riemannian-like'' metric spaces
(see Section~\ref{QCtoSC}),
and the rectifiability of gradient curves is an intriguing problem in optimization theory
and dynamical systems.
We refer to \cite{DLS,DDDL} for more about
the background information on self-contracted curves,
and\cite[Section~6]{DLS} for further examples of self-contracted curves.

Recall that a curve $\gg:I \longrightarrow X$ is said to be rectifiable if 
\[ \sup\bigg\{\sum_{i=1}^{k-1}d\big( \gg(t_i), \gg(t_{i+1}) \big)
 \,\bigg\vert\, t_1,\dots,t_{k} \in I,\, t_1 < \dots < t_{k} \bigg\} < \infty. \]
It was shown that bounded self-contracted curves are rectifiable,
on the one hand, in the following spaces:
\begin{itemize}
\item{Euclidean spaces \cite{DDDL, LMV},}
\item{Riemannian manifolds \cite{DDDR},}
\item{finite-dimensional normed spaces \cite{Le, ST},}
\item{$\CAT(0)$-spaces with some additional properties \cite{OhRect}.}
\end{itemize}
These results are more or less based on the fundamental construction in the Euclidean case.
Precisely, such a construction was generalized to normed spaces and $\CAT(0)$-spaces,
and the case of Riemannian manifolds was shown by approximations and a compactness argument.

On the other hand, an observation in \cite{ZSC} revealed an interesting connection
between self-contracted curves and an embedding problem of finite metric spaces,
and the following spaces were added to the above list:
\begin{itemize}
\item{complete Finsler manifolds,}
\item{complete, locally compact $\CAT(k)$-spaces ($k \in \R$) with locally extendable geodesics,}
\item{complete, locally compact Busemann NPC spaces with locally extendable geodesics,}
\item{complete, locally compact geodesic spaces with locally extendable geodesics
satisfying the local rough Busemann condition}
\end{itemize}
(see \cite{ZSC} for the precise meaning of those conditions on geodesics and curvature bounds).
Let us stress that all the above results are on spaces with \emph{upper curvature bounds}.

The aim of the present paper is to show that the strategy in \cite{ZSC}
also applies to spaces with \emph{lower curvature bounds}.
As a result we further extend the above list with:
\begin{itemize}
\item{finite-dimensional Alexandrov spaces of curvature $\ge k$ with $k \in \R$
(Proposition~\ref{CBBs-are-ATB}),}
\item{complete Berwald spaces of nonnegative flag curvature
(Proposition~\ref{Ber-are-ATB}).}
\end{itemize}
These are consequences of our main result below (Theorem~\ref{ATB-no-SRA}).
Our approach also provides an alternative proof for normed spaces.
We remark that the finiteness of the dimension is essential
since one can easily construct counter-examples in Hilbert spaces
(see \cite[Example~2.2]{DDDR}, \cite[Example~2.6]{OhRect}).

\subsection{Small rough angle condition}

As in \cite{ZSC},
the following condition inspired by \cite{BS} will play an essential role.

\begin{definition}[Small rough angle condition; $\SRA(\aa)$]
Let $(X,d)$ be a metric space and $0 < \aa < 1$.
We say that $X$ satisfies the \emph{$\SRA(\aa)$-condition} if,
for every $x,y,z \in X$, we have 
\begin{equation}
d(x,y) \le \max\{d(x,z) + \aa d(z,y), \aa d(x,z) + d(z,y)\}. \label{SRA-ineq}
\end{equation}
\end{definition}

Note that \eqref{SRA-ineq} implies, provided $x,y \in X \setminus \{z\}$,
\begin{equation}
\wt\angle x z y \le \arccos(-\aa) =  \pi - \arccos(\aa), \label{SRA-Angles}
\end{equation}
where $\wt\angle xzy \in [0,\pi]$ denotes the \emph{Euclidean comparison angle}
defined as the angle $\angle \tilde{x} \tilde{z} \tilde{y}$
for the comparison triangle $\triangle \tilde{x} \tilde{y} \tilde{z} \subset \R^2$
having the same side lengths as $\triangle xyz \subset X$.
By the Euclidean cosine formula we have
\begin{equation}\label{eq:cosine}
\cos \wt\angle xzy =\frac{d(x,z)^2 +d(y,z)^2 -d(x,y)^2}{2d(x,z)d(y,z)}.
\end{equation}
Indeed, if the opposite inequality 
$\wt\angle xzy > \pi-\arccos(\aa)$ to \eqref{SRA-Angles} holds,
then we have
\[ \aa < \cos(\pi - \wt\angle xzy) =-\cos \wt\angle xzy \]
and \eqref{eq:cosine} implies
\begin{align*}
d(x,y)^2 &> d(x,z)^2 +d(y,z)^2 +2\alpha d(x,z) d(y,z) \\
 &> (\max\{d(x,z) + \aa d(z,y), \aa d(x,z) + d(z,y)\} )^2.
\end{align*}

In \cite{BS} the condition $\SRA(\aa)$ was used
to study the non-embedding of snowflakes (see the next subsection).
A connection with self-contracted curves was established in \cite[Theorem~1]{ZSC} as follows.
 
\begin{theorem}[$\SRA$-sets in self-contracted curves, \cite{ZSC}] \label{SRAinSC}
Assume that a metric space $(X,d)$ admits
a bounded, self-contracted, and unrectifiable curve $\gg:I \longrightarrow X$.
Then, for any $1/2 <\aa<1$ and $k \in \N$,
there exists a $k$-point subset $Y \subset \gg(I)$ satisfying the $\SRA(\aa)$-condition.
\end{theorem}

We remark that, if $\gamma(s)=\gamma(t)$ for $s<t$,
then $\gamma$ is constant on $[s,t]$ by the self-contractedness (see \cite[Lemma~2.1]{OhRect}).
Thus the unrectifiability implies that the image of $\gamma$ is indeed an infinite set.
Theorem~\ref{SRAinSC} asserts that, in other words,
if a metric space does not admit any large subset satisfying $\SRA(\aa)$,
then every bounded self-contracted curve in it is rectifiable.
Along this strategy of \cite{ZSC},
we introduce metric spaces with weak lower curvature bound
and show that they do not contain large $\SRA(\aa)$-subsets.

\subsection{Angular total boundedness and main theorems}

We denote by $B_R(x)$ the open ball of center $x$ and radius $R$.

\begin{definition}[Angular total boundedness; $\ATB(\ee)$]\label{ATB}
Let $(X,d)$ be a metric space and $0<\ee<\pi/2$.
We say that a point $p \in X$ satisfies the \emph{$\ATB(\ee)$-condition}
if there exist some $L \in \N$ and $R > 0$ such that,
for every $y_1,\dots,y_L \in B_R(p) \setminus \{p\}$,
we can find $i \neq j$ satisfying
\[ \wt\angle y_i p y_j < \ee. \]
We say that $X$ is a \emph{globally $\ATB(\ee)$-space}
if there exists $L \in \N$ such that every $p \in X$ satisfies the $\ATB(\ee)$-condition
with constants $L$ and $R = \infty$.
We say that $X$ is a \emph{semi-globally $\ATB(\ee)$-space} if,
for any $x \in X$ and  $R>0$, there exists $L \in \N$ 
such that every $p \in B_R(x) $ satisfies the $\ATB(\ee)$-condition with constants $L$ and $R$.
Finally we say that $X$ is a \emph{locally $\ATB(\ee)$-space} if, for any $x \in X$,
there exist $L \in \N$ and $R > 0$ such that every $p \in B_R(x) $ satisfies the $\ATB(\ee)$-condition
with constants $L$ and $R$.
\end{definition}

Roughly speaking, the $\ATB(\ee)$-condition asserts
the total boundedness with respect to the comparison angle.
Examples of locally $\ATB(\ee)$-spaces include:
\begin{itemize}
\item{Finsler manifolds (Proposition~\ref{Fin-are-ATB}),}
\item{complete, locally compact $\CAT(k)$-spaces ($k \in \R$) with locally extendable geodesics
(Proposition~\ref{CATExt-are-ATB}(\ref{CATk})).} 
\end{itemize}
Next we have semi-globally $\ATB(\ee)$-spaces:
\begin{itemize}
\item{finite-dimensional Alexandrov spaces of curvature $\ge k$ with $k \in \R$
(Proposition~\ref{CBBs-are-ATB}),}
\item{complete, locally compact Busemann NPC spaces 
(e.g., $\CAT(0)$-spaces) with locally extendable geodesics
(Proposition~\ref{CATExt-are-ATB}(\ref{B-NPC})).}
\end{itemize}
Finally, globally $\ATB(\ee)$-spaces include:
\begin{itemize}
\item{finite-dimensional Alexandrov spaces of nonnegative curvature
(Proposition~\ref{CBBs-are-ATB}),}
\item{finite-dimensional normed spaces
(Proposition~\ref{pr:norm}),}
\item{complete Berward spaces of nonnegative flag curvature
(Proposition~\ref{Ber-are-ATB}),}
\item{Cayley graphs of virtually abelian groups
(Proposition~\ref{virt-abel}).}
\end{itemize}

Notice that the class of $\ATB(\ee)$-spaces is so large
that it admits both lower and upper curvature bounds,
as well as both Riemannian and Finsler spaces.

Our main theorem is the following.

\begin{theorem}[$\ATB$ prevents large $\SRA$-sets]\label{ATB-no-SRA}
Let $(X,d)$ be a metric space, $0<\aa<1$ and $\ee :=\arccos(\aa)/2 > 0$.
Given $x \in X$ and $R > 0$, suppose that there is $L \in \N$ such that
every $p \in B_R(x)$ satisfies the $\ATB(\ee)$-condition with constants $L$ and $2R$.
Then there exists $N(L) \in \N$ such that there is no $N$-point $\SRA(\aa)$-subset in $B_R(x)$.   
\end{theorem}

We give the proof of Theorem~\ref{ATB-no-SRA} in Section~\ref{ATB-no-SRA-proof}. 
Combining Theorem~\ref{ATB-no-SRA} with Theorem~\ref{SRAinSC},
we obtain the following.
A metric space is said to be \emph{proper} if every bounded closed set is compact.

\begin{theorem}[Rectifiability of self-contracted curves]\label{ATB-to-SC}
Let $0<\ee<\pi/6$. If $(X,d)$ is a  semi-globally $\ATB(\ee)$-space or a 
proper locally $\ATB(\ee)$-space,
then every bounded self-contracted curve in $X$ is rectifiable.
\end{theorem}


The following aforementioned results are now given as a corollary.

\begin{corollary}\label{BBC-to-SC}
Bounded self-contracted curves in finite-dimensional Alexandrov spaces of curvature $\ge k$
or in complete Berwald spaces of nonnegative flag curvature are rectifiable.
\end{corollary}

\subsection{Non-embeddability of snowflakes}

The $\SRA(\aa)$-condition finds applications in the embedding problem
via the following observation found in \cite[(3.3)]{BS}.
Given a metric space $(Y,d)$ and $0<\aa<1$,
we denote by $(Y,d^{\aa})$ its \emph{$\aa$-snowflake}
defined by $d^{\aa}(x,y):=d(x,y)^{\aa}$.

\begin{proposition}[Snowflakes satisfy $\SRA$, \cite{BS}]\label{SAinSN}
For any metric space $(Y,d)$ and $0<\aa<1$,
the $\aa$-snowflake $(Y,d^{\aa})$ satisfies the $\SRA(\aa)$-condition.
\end{proposition}
 
Combining this with Theorem~\ref{ATB-no-SRA} leads the following.

\begin{theorem}[$\ATB$ prevents large snowflakes]\label{SnowComp}
Let $(X,d)$ be a metric space, $0<\aa<1$ and $\ee:=\arccos(\aa)/2$.
Given $x \in X$ and $R > 0$, suppose that
there is $L \in \N$ such that
every $p \in B_R(x)$ satisfies the $\ATB(\ee)$-condition with constants $L$ and $2R$.
Then there exists $N =N(L) \in \N$ such that,               
if a metric space $(Y,d_Y)$ has cardinality at least $N$,
then its $\alpha$-snowflake $(Y,d_Y^{\aa})$ does not admit an isometric embedding into $B_R(x)$.
\end{theorem}

As a corollary, we obtain restrictions on the isometric embeddability of large snowflakes
into Alexandrov spaces or Berwald spaces of nonnegative flag curvature.
Here we state the former for completeness.

\begin{corollary}\label{BBC-no-SF}
Let $(X,d)$ be an $n$-dimensional Alexandrov space of curvature $\ge k$ with $k \in \R$.
Then, for any $0<\aa<1$ and $R>0$, there exists $\NNN = \NNN(n,k,\aa,R) \in \N$ such that,               
if a metric space $(Y,d_Y)$ has cardinality at least $\NNN$,
then its $\alpha$-snowflake $(Y,d_Y^{\aa})$ does not admit an isometric embedding into any $R$-ball in $X$.
Moreover, if $k=0$, then we can take $R=\infty$ and $\NNN=\NNN(n,\aa)$.
\end{corollary}

This article is organized as follows.
In Section~\ref{ATB-no-SRA-proof}, we prove Theorems~\ref{ATB-no-SRA} and \ref{ATB-to-SC}.
In Section~\ref{SecCAlemma}, we introduce a technical version of the $\ATB(\ee)$-condition
named the \emph{$\ATB^*(\ee)$-condition}, which is useful for proving $\ATB(\ee)$.
With the help of the tools introduced in Section~\ref{sc:Tools},
we give examples of locally, semi-globally, and globally $\ATB(\ee)$-spaces in Section~\ref{Examples}.
Section~\ref{QCtoSC} is devoted to the self-contractedness of gradient curves
of quasi-convex functions.
In Section~\ref{sc:Doubling}, we show that
the absence of large $\SRA(\aa)$-subsets implies the doubling condition,
as a weak converse to Theorem~\ref{ATB-no-SRA}.
Finally, in Sections~\ref{Counter} and \ref{Questions},
we discuss several intriguing (counter-)examples and open questions, respectively.

\section{Proofs of Theorems~\ref{ATB-no-SRA} and \ref{ATB-to-SC}} \label{ATB-no-SRA-proof}

The proof of Theorem~\ref{ATB-no-SRA} makes use of the Ramsey theory.
For $a,b \in \N$, we denote by $\RRR(a,b)$ the corresponding \emph{Ramsey number}.
That is to say, when we color the edges of the complete graph of $\RRR(a,b)$-vertices in two colors,
then there exists either an $a$-point subset of vertices such that all the edges
between these vertices are in the first color,
or a $b$-point subset of vertices such that all the edges
between these vertices are in the second color.

\begin{proof}[Proof of Theorem~$\ref{ATB-no-SRA}$]
The hypotheses of the theorem provides that,
for every $p \in B_R(x)$ and $y_1,\dots,y_L  \in B_R(x) \setminus \{p\}$,
there exists $1 \le i < j \le L$ such that 
\begin{equation}
\wt\angle y_i p y_j < \ee.\label{FDPC}
\end{equation}
Let us define
\[ H_{L} := \RRR(2, L), \quad
 H_{L - 1} := \RRR(H_{L} + 1, L), \quad \ldots, \quad
 H_2 := \RRR(H_3 + 1, L). \]
We shall show that we can take $N = H_2 + 1$.
By contradiction let $Z$ be an $N$-point $\SRA(\aa)$-subset in $B_R(x)$. 
Take a point $z_1$ from $Z$ and
consider the complete graph of the remaining $H_2$ points. 
We color edges of the graph in two colors.
We color an edge $\{v, w\}$ in the first color if
\[ \wt\angle v z_1 w < \ee, \]
and in the second otherwise.
By \eqref{FDPC} there is no $L$-point subset such that
all edges in the corresponding subgraph are of the second color.
Thus, by the choice of $H_2=\RRR(H_3+1,L)$,
there exists an $(H_3 + 1)$-point subset such that
all edges in the corresponding subgraph are of the first color.
We denote this set of $(H_3 +1)$ points by $\wt Z_2$, and take some $z_2 \in \wt Z_2$.
Next we set $Z_2:=\wt Z_2 \setminus \{z_2\}$
and consider the coloring by $\wt\angle v z_2 w <\ee$ or not,
find an $(H_4 +1)$-point set $\wt Z_3 \subset Z_2$ and take $z_3 \in \wt Z_3$ in the same way.
Put $Z_3:=\wt Z_3 \setminus \{z_3\}$ and iterate this procedure.
In the last step we find a $2$-point set $\wt Z_L$ and
let us denote it by $\wt Z_L =\{z_L,z_{L+1}\}$.

On the one hand,
by the definition of the first color we have $\wt \angle z_j z_i z_k < \ee$
for any $1 \le i < j < k \le L+1$.
This implies that 
\[ \wt \angle z_i z_k z_j + \wt \angle z_i z_j z_k  = \pi - \wt \angle z_j z_i z_k  > \pi - \ee. \]
On the other hand, we deduce from the $\SRA(\aa)$-condition \eqref{SRA-Angles} that
\[ \wt \angle z_i z_j z_k \le \pi -\arccos(\aa) = \pi - 2\ee. \]
Combining the last two inequalities provides
\[ \wt \angle z_i z_k z_j > \ee \]
for every  $1 \le i < j < k \le L+1$.
Therefore, if we take $p = z_{L+1}$ and  $y_i = z_i$ for $i = 1,\dots, L$,
then we have a contradiction with \eqref{FDPC}.
This completes the proof.
\end{proof}

Theorem~\ref{ATB-to-SC} is almost straightforward,
here we give a proof for completeness.

\begin{proof}[Proof of Theorem~$\ref{ATB-to-SC}$]
The semi-global case is immediate,
given a bounded, self-contracted and unrectifiable curve $\gamma$,
one can find an arbitrary large $\SRA(\alpha)$-set in the image of $\gamma$ 
with $\alpha=\cos(2\ee) \in (1/2,1)$ by Theorem~\ref{SRAinSC}.
This contradicts Theorem~\ref{ATB-no-SRA}.

In the local case, let us consider the covering $\{ B_{R_x}(x) \}_{x \in X}$ of $X$,
where every $p \in B_{R_x}(x)$ satisfies the $\ATB(\ee)$-condition 
for some $L_x \in \N$ and $R_x>0$.
Since $X$ is proper and $\gamma$ is bounded, the image of $\gamma$
is covered by finite elements of $\{ B_{R_x}(x) \}_{x \in X}$.
Then, if $\gamma$ is not rectifiable,
the restriction of $\gamma$ to one of those balls is unrectifiable.
This implies a contradiction similarly to the previous argument.
\end{proof}

\section{$\ATB^*(\ee)$-condition}\label{SecCAlemma}


The aim of this section is to introduce a technical version 
of the $\ATB(\ee)$-condition, named the $\ATB^*(\ee)$-condition. 
The following lemma connects these conditions.
For $0 < \ee < \pi/2$, we define
\begin{equation}\label{bb-prop}
\bb(\ee) := \frac{(1 -\cos\ee) \sin \ee}{2(1+\sin\ee)}.
\end{equation}
Notice that $\bb(\ee)<1/4$ and $\bb(\ee)=O(\ee^3)$ as $\ee \to 0$.

\begin{lemma}\label{CAlemma}
Let $(X,d)$ be a metric space, and take $p \in X$ and $x,y \in X \setminus \{p\}$.
If there is $x' \in X$ such that $d(p,y) =d(p,x') +d(x',y)$
and $d(x,x') \le \beta(\ee)d(p,x)$, then we have $\wt{\angle} xpy < \ee$.
\end{lemma}

\begin{proof}
First of all, if $d(x, y) < d(p, x)\sin\ee$, then we find (recalling \eqref{eq:cosine})
\[ \cos\wt{\angle} xpy =\frac{d(p,x)^2 +d(p,y)^2 -d(x,y)^2}{2d(p,x)d(p,y)}
 > \frac{\cos^2 \ee \cdot d(p,x)^2 +d(p,y)^2}{2d(p,x)d(p,y)}. \]
Thus Young's inequality shows $\cos\wt{\angle} xpy > \cos\ee$
and hence $\wt{\angle} xpy < \ee$ (without the hypothesis on the existence of $x'$).

Assume $d(x, y) \ge d(p, x)\sin\ee$ (and hence $x \neq y$).
We deduce from the triangle inequality and the hypotheses that
\begin{align}
d(p, y) &= d(p, x') + d(x', y) \ge d(p, x) + d(x, y) - 2d(x,x') \nonumber\\
 &\ge d(x, y) + (1 - 2\bb)d(p, x) \label{dc213}
\end{align}
with $\bb=\bb(\ee)$ for simplicity.
Substituting
\[ d(p,y)^2 =d(p,x)^2 +d(x,y)^2 -2d(p,x) d(x,y) \cos\wt{\angle} pxy \]
into the square of \eqref{dc213} yields
\begin{align*}
&d(p, x)^2 -2d(p,x)d(x,y) \cos\wt{\angle} pxy \\
&\ge (1 - 2\bb)^2 d(p, x)^2   + 2(1 - 2\bb)d(x, y)d(p, x).
\end{align*}
After regrouping we have 
\[ \cos\wt{\angle} pxy \le (2\bb-1) +\frac{2d(p, x)}{d(x, y)} (\bb-\bb^2). \]
Then it follows from the assumption $d(x, y) \ge d(p, x)\sin\ee$
and the definition \eqref{bb-prop} of $\bb=\bb(\ee)$ that
\[ \cos\wt{\angle} pxy \le (2\bb-1) +\frac{2(\bb-\bb^2)}{\sin\ee}
 < -1 +\frac{2(1+\sin\ee)}{\sin\ee}\bb =-\cos\ee. \]
Therefore $\wt{\angle} pxy > \pi - \ee$ and hence $\wt{\angle} xpy < \ee$.
This completes the proof.
\end{proof}

Let $X$ be a metric space and $\GG$ be some set of minimizing geodesics in $X$. 
We say that $\GG$ is a \emph{quasi-bicombing on $X$}
if for every $x,y \in X$ there is $\gg \in \GG$ with endpoints $x$ and $y$
(in particular, $X$ is a geodesic space).

\begin{definition}[$\ATB^*(\ee)$-condition]\label{ATB*}
Let $(X,d)$ be a metric space, $\GG$ be a quasi-bicombing on $X$ and $0<\ee<\pi/2$.
We say that $p \in X$ satisfies the \emph{$\ATB^*(\ee)$-condition with respect to $\GG$}
if there are $L \in \N$ and $R > 0$ such that,
for every $y_1,\dots,y_L \in B_R(p) \setminus \{p\}$ and
any $\gg_1,\dots, \gg_L \in \GG$ where $\gg_i$ is connecting $p$ and $y_i$,
we can find $i \neq j$ satisfying 
\[ \min_{x \in \gg_j} d(y_i,x) \le \bb(\ee)d(p, y_i), \]
for $\bb(\ee)$ defined in \eqref{bb-prop}.
\end{definition}

It follows from Lemma~\ref{CAlemma} that, if a point $p$ in a geodesic space
satisfies the $\ATB^*(\ee)$-condition  with respect to some quasi-bicombing
for constants $L \in \N$ and $R > 0$,
then it also satisfies the $\ATB(\ee)$-condition with the same constants. 
An advantage of the $\ATB^*(\ee)$-condition will be seen in Lemma~\ref{LiftLemma}(\ref{Adv}).

\section{Tools for proving $\ATB(\ee)$} \label{sc:Tools}

In this section we provide two lemmas those will be our main ingredients
for proving the $\ATB(\ee)$-condition for all the examples we discuss.
The first lemma will be used to study, for example, the lower curvature bounds.
The second lemma will be used in the upper curvature bounds. 

\begin{lemma}\label{LiftLemma}
Let $(X,d_X)$ and $(Y,d_Y)$ be metric spaces, $p \in X$, $q \in Y$, $R > 0$,
and $\Phi:B_R(q) \longrightarrow X$ be a map with $\Phi(q) = p$.
\begin{enumerate}[{\rm (1)}]
\item{\label{Simp}
\emph{The simple version}.
Suppose that $\Phi$ satisfies the following conditions:
\begin{enumerate}[{\rm (a)}]
\item{\label{SimpRad}$d_X(p, \Phi(y)) = d_Y(q, y)$ for every $y \in B_R(q)$,}
\item{\label{SimpAll}$d_X(\Phi(y),\Phi(y')) \ge d_Y(y,y')$ for every $y,y' \in B_R(q).$}
\end{enumerate}
If there are $L \in \N$ and $0<\ee<\pi/2$ such that
$p$ satisfies the $\ATB(\ee)$-condition for $L$ and $R$,
then $q$ also satisfies the $\ATB(\ee)$-condition for the same $L$ and $R$.  
}
\item{\label{Adv}
\emph{The $\ATB^*$-version}.
Suppose that there are quasi-bicombings $\GG_X$ and $\GG_Y$ on $X$ and $Y$
satisfying the following conditions:}
\begin{enumerate}[{\rm (a)}]
\item{\label{StarRad}$d_X(p, \Phi(y)) \le K_1 d_Y(q, y)$
for some $K_1 > 0$ and every $y \in B_R(q)$,}
\item{\label{StarAll}$d_X(\Phi(y),\Phi(y')) \ge K_2 d_Y(y,y')$
for some $K_2 > 0$ and every $y,y' \in B_R(q)$,}
\item{\label{StarGG}For every $\gg \in \GG_Y$ having $q$ as an endpoint,
we have $\Phi(\gg) \in \GG_X$.}
\end{enumerate} 
If there are $L \in \N$ and $0<\ee<\pi/2$ with $\beta(\ee)<K_2/4K_1$
such that $p$ satisfies the $\ATB^*(\ee')$-condition with respect to $\GG_X$ for $L$ and $K_1R$, 
then $q$ satisfies the $\ATB^*(\ee)$-condition with respect to $\GG_Y$ for $L$, $R$,
and $0<\ee<\pi/2$ satisfying $\bb(\ee') \le K_2 \bb(\ee)/K_1$.
\end{enumerate}   
\end{lemma}

\begin{proof}
(\ref{Simp})
Let $y_1,\dots,y_L \in B_R(q) \setminus \{q\}$.
Note that $\Phi(y_i) \neq p$ by (\ref{SimpRad}).
Then the $\ATB(\ee)$-condition for $p$ gives $i \neq j$ such that 
\[ \wt\angle \Phi(y_i) p \Phi(y_j) < \ee. \]
Combining this with the hypotheses (\ref{SimpRad}) and (\ref{SimpAll}) implies 
\[ \wt\angle y_i q y_j < \ee. \]
This completes the proof of the simple version.

(\ref{Adv})
Let $y_1,\dots,y_L \in B_R(q) \setminus \{q\}$ and take
$\gg_1,\dots, \gg_L \in \GG_Y$ such that $\gg_i$ connects $q$ and $y_i$.
By (\ref{StarGG}) we have $\Phi(\gg_1),\dots, \Phi(\gg_L) \in \GG_X$,
and (\ref{StarRad}) and (\ref{StarAll}) ensure $\Phi(y_i) \in B_{K_1 R}(p) \setminus \{p\}$. 
The $\ATB^*(\ee')$-condition for $p$ yields $i \neq j$ such that
\begin{equation}
\min_{x \in \Phi(\gg_j)} d_X\big( \Phi(y_i),x \big) \le \bb(\ee')d_X\big( p, \Phi(y_i) \big).
\label{ATBs-in}\end{equation}
Fix $x \in \Phi(\gg_j)$ for which this minimum is achieved,
and take $y \in \gg_j$ such that $\Phi(y) =x$.
By (\ref{StarAll}) and (\ref{ATBs-in}) we have 
\[ d_Y(y_i,y) \le \frac{1}{K_2}d_X\big( \Phi(y_i), x \big)
 \le \frac{\bb(\ee')}{K_2}d_X\big( p, \Phi(y_i) \big). \]
Moreover, we deduce from (\ref{StarRad}) that
\[ \frac{\bb(\ee')}{K_2}d_X\big( p, \Phi(y_i) \big) 
 \le \frac{\bb(\ee')K_1}{K_2}d_Y(q, y_i) \le \bb(\ee) d_Y(q, y_i). \]
Thus we have  
\[ \min_{y' \in \gg_j} d_Y(y_i,y') \le \bb(\ee)d_Y(q, y_i) \]
and this completes the proof.
\end{proof}

\begin{definition}[Locally extendable geodesics]\label{defLEG}
Let $(X,d)$ be a geodesic metric space.
We say that $p \in X$ satisfies the \emph{locally extendable geodesics} condition
with a parameter $\dd > 0$ if every unit speed minimizing geodesic
$\gg: [0,T] \longrightarrow \wb{B_\dd(p)}$
can be extended to a unit speed minimizing geodesic $\wt \gg: [0,T + \dd] \longrightarrow X$,
i.e., $\wt \gg(t) = \gg(t)$ for all $t \in [0,T]$.
We say that $X$ is a space with locally extendable geodesics
if every $p \in X$ satisfies the locally extendable geodesics condition with some $\dd(p) > 0$.
\end{definition}

\begin{definition}[Local rough Busemann condition; $\LRB$]\label{defLRB}
Let $(X,d)$ be a geodesic metric space.
We say that $p \in X$ satisfies the \emph{local rough Busemann} condition,
the $\LRB$-condition for short, with parameters $H, K > 0$
if, for every pair of unit speed minimizing geodesics
$\gg_1: [0,L_1] \longrightarrow  \wb{B_H(p)}$ and $\gg_2: [0,L_2] \longrightarrow  \wb{B_H(p)}$
with $\gg_1(0) = \gg_2(0)$, we have 
\[ d\big( \gg_1(tL_1),\gg_2(tL_2) \big) \le Ktd\big( \gg_1(L_1),\gg_2(L_2)\big) \]
for all $t \in [0,1]$.
\end{definition}

A subset of a metric space is said to be \emph{$r$-separated}
if every pair of distinct points in that set is of distance $\ge r$.
We will denote by $S_R(p)$ the sphere of center $x$ and radius $R$,
namely the set of points $y$ of $d(p,y)=R$.

\begin{lemma}\label{ATBs-from-LRB}
Let $X$ be a geodesic metric space, $R > 0$, and $p \in X$ be a point satisfying
the locally extendable geodesics condition with $\dd = R$ 
and the $\LRB$-condition with $H = R$ and some $K > 0$. 
Let $0<\ee<\pi/2$ and suppose that the cardinality of any $(\bb(\ee)R/K)$-separated set
in $S_R(p)$ is less than $L \in \N$.
Then $p$ satisfies the $\ATB^*(\ee)$-condition with respect to any quasi-bicombing $\GG$ on $X$
for $L$ and $R$.
\end{lemma}

We remark that in Lemma~\ref{ATBs-from-LRB} everything happens in $\wb{B_R(p)}$,
where the $\LRB$-condition ensures the uniqueness of a minimizing geodesic
connecting any pair of points. 
Thus nothing really depends on the quasi-bicombing $\GG$.

\begin{proof}
Let  $y_1, \dots, y_{L} \in B_R(p) \setminus \{p\}$ 
and $\gg_1, \dots, \gg_L$ be the unique minimizing geodesics connecting $p$ and these points. 
We need to find $i \neq j$ such that 
\begin{equation}
\min_{x \in \gg_j} d(y_i,x) \le \bb(\ee)d(p, y_i).
\label{ATBs-eq}\end{equation}
By contradiction suppose that \eqref{ATBs-eq} does not hold for any $i \neq j$.
For $1 \le k \le L$, let $\wb y_k$ be the intersection of an extension of $\gg_k$ beyond $y_k$
with the sphere $S_R(p)$, provided by the locally extendable geodesics condition.
Then we shall see that 
\begin{equation}
d(\wb y_i, \wb y_j) > \frac{\bb(\ee)R}{K}
\label{Far-on-Sp}\end{equation}
holds for all $i \neq j$.
Let $d(p, y_i) \le d(p, y_j)$ without loss of generality
and take $x \in \gg_j$ such that $d(p, y_i) = d(p, x)$.
We deduce from the negation of \eqref{ATBs-eq} that $d(y_i, x) > \bb(\ee)d(p,y_i)$.
Then the $\LRB$-condition implies \eqref{Far-on-Sp} as
\[ d(\wb y_i, \wb y_j) \ge \frac{R}{Kd(p,y_i)}d(y_i,x) > \frac{\bb(\ee)R}{K}. \]
This contradicts the hypothesis on the cardinality of $(\bb(\ee)R/K)$-separated sets,
and completes the proof.
\end{proof}




\section{Examples of $\ATB(\ee)$-spaces}\label{Examples}

We present examples of metric spaces satisfying 
the $\ATB(\ee)$-condition locally, semi-globally or globally 
(recall Definition~\ref{ATB} for the definition).
The $\ATB(\ee)$-condition turns out flexible and covers both 
Riemannian and Finsler spaces,
as well as both lower and upper curvature bounds.

\subsection{Metric spaces with upper curvature bound}

We start with metric spaces with upper curvature bounds.
We refer to \cite{BBI} for the basics of \emph{$\CAT(k)$-spaces}
(geodesic metric spaces of sectional curvature $\le k$).
Note that $\CAT(0)$-spaces are \emph{nonpositively curved in the sense of Busemann}
(\emph{Busemann NPC} for short).

\begin{proposition}[$\ATB$ of $\CAT(k)$-spaces]\label{CATExt-are-ATB}
\begin{enumerate}[{\rm (1)}]
\item\label{CATk}
For any $k \in \R$ and $0<\ee<\pi/2$, complete,
locally compact $\CAT(k)$-spaces with locally extendable geodesics are locally $\ATB(\ee)$.

\item\label{B-NPC}
For any $0<\ee<\pi/2$, complete,
locally compact Busemann NPC spaces with locally extendable geodesics are semi-globally $\ATB(\ee)$.
\end{enumerate}
\end{proposition}

\begin{proof}
(\ref{CATk})
Let $(X,d)$ be a locally compact $\CAT(k)$-space satisfying the locally extendable geodesics
condition with $\dd>0$.
Fix $x \in X$.
By \cite[Section 3]{OhtaConv}, there exist $0 < R \le \dd/2$ and $K > 0$
such that every $p \in B_R(x)$ satisfies the $\LRB$-condition for $H = R$ and $K$.
Since $X$ is a complete, locally compact geodesic space,
it is proper and hence $\wb{B_{2R}(x)}$ is compact
(see \cite[Proposition 2.5.22]{BBI}).
Thus there exists $L \in \N$ such that the cardinality of any $(\bb(\ee)R/K)$-separated set
in $\wb{B_{2R}(x)}$ is less than $L$. 
Then Lemma~\ref{ATBs-from-LRB} shows that $p$ satisfies the $\ATB^*(\ee)$-condition 
with respect to any quasi-bicombing $\GG$ on $X$ for $L$ and $R$.
We finally apply Lemma~\ref{CAlemma} to see that 
$p$ satisfies the $\ATB(\ee)$-condition with the same parameters $L$ and $R$.

(\ref{B-NPC})
In this case $R$ can be taken arbitrarily large and we have the semi-global $\ATB(\ee)$
by the same discussion.
\end{proof}

We present two results as a corollary for later use.
For $k < 0$ and $n \in \N$, we denote by $\mathbb{H}^n(k)$ the $n$-dimensional hyperbolic space.
In a normed space, we call the \emph{linear quasi-bicombing}
the one which consists of all line segments.

\begin{corollary}\label{Hnk-ATB}
\begin{enumerate}[{\rm (1)}]
\item\label{Hnk-ATB-Hn}
For $k < 0$, $n \in \N$, $R > 0$ and $0 < \ee < \pi/2$,
there exists $L=L(k,n,R,\ee) \in \N$ such that every point $p \in \mathbb{H}^n(k)$ satisfies
the $\ATB(\ee)$-condition for $L$ and $R$. 

\item\label{Hnk-ATB-Rn}
For $n \in \N$ and $0 < \ee < \pi/2$,
there exists $L=L(n,\ee) \in \N$ such that every point $p$ in the Euclidean space $\R^n$ satisfies
the $\ATB^*(\ee)$-condition with respect to the linear quasi-bicombing for $L$ and $R = \infty$.   
\end{enumerate}
\end{corollary} 

\begin{proof}
The first assertion is immediate from Proposition~\ref{CATExt-are-ATB}(\ref{B-NPC}).
The second assertion follows from Lemma~\ref{ATBs-from-LRB}
by noticing $K=1$ and the fact that
the cardinality of $(\bb(\ee)R)$-separated sets in $\wb{B_R(p)}$
is bounded above independent of $R$.
\end{proof}

\subsection{Metric spaces with lower curvature bound}

Next we consider finite-dimensional \emph{Alexandrov spaces of curvature bounded below}.
We again refer to \cite{BBI} for the basics of those spaces.



\begin{proposition}[$\ATB$ of Alexandrov spaces]\label{CBBs-are-ATB}
For any $k \in \R$ and $n \in \N$,
every $n$-dimensional Alexandrov space of curvature $\ge k$
satisfies the semi-global $\ATB(\ee)$-condition for any $0<\ee<\pi/2$.
Moreover, if $k \ge 0$, then it satisfies the global $\ATB(\ee)$-condition
for any $0<\ee<\pi/2$.
\end{proposition}

\begin{proof}
In the first assertion, it suffices to consider the case of $k<0$.
Let $(X,d)$ be an $n$-dimensional Alexandrov space of curvature $\ge k$,
and compare $X$ with the hyperbolic space $\mathbb{H}^n(k)$.
Fix $R>0$ and $0<\ee<\pi/2$.
Recall from Corollary~\ref{Hnk-ATB}(\ref{Hnk-ATB-Hn}) that
there exists $L \in \N$ such that every point in $\mathbb{H}^n(k)$
satisfies the $\ATB(\ee)$-condition for $L$ and $R$. 

We claim that every point $p \in X$ satisfies the $\ATB(\ee)$-condition for $L$ and $R$. 
If not, then (by the openness of the condition)
any point $x \in B_{\ww}(p)$ with sufficiently small $\ww > 0$
does not satisfy the $\ATB(\ee)$-condition for $L$ and $R$.
Note that almost every point in $B_{\ww}(p)$ has the Euclidean tangent cone,
and let us fix $\bar{x}$ as one of such points.
Now we define the map $\Phi:B_R(\bar{x}) \longrightarrow T_{\bar{x}}X$ as follows,
where $T_{\bar{x}}X$ is equipped with the hyperbolic metric
(thus isometric to $\mathbb{H}^n(k)$).
For $x \in B_R(\bar{x})$, let $\gg_x$ be a minimizing geodesic from $\bar{x}$ to $x$,
and define $\Phi(x):=v_x \in T_{\bar{x}}X$ as the velocity vector of $\gg_x$ with length $d(\bar{x},x)$.
By the monotonicity of the comparison angle in Alexandrov spaces,
$\Phi$ satisfies the conditions of Lemma~\ref{LiftLemma}(\ref{Simp}). 
Therefore $\bar{x}$ satisfies the $\ATB(\ee)$-condition for $L$ and $R$.
This is, however, a contradiction and completes the proof for $k<0$.

In the case of $k=0$, the claim follows from the same discussion and
Corollary~\ref{Hnk-ATB}(\ref{Hnk-ATB-Rn}).
\end{proof}



\subsection{Finite-dimensional normed spaces}

Recall that, in a normed space,
the linear quasi-bicombing means the one which consists of all line segments.

\begin{proposition}[ATB of normed spaces]\label{pr:norm}
For $n \in \N$ and $0<\ee<\pi/2$, there exists $L_N=L_N(n,\ee)$ such that
every point in any $n$-dimensional normed space satisfies the $\ATB^*(\ee)$-condition
with respect to the linear quasi-bicombing for $L_N$ and $R = \infty$.
In particular, every $n$-dimensional normed space satisfies
the global $\ATB(\ee)$-condition for $L_N$.
\end{proposition}

\begin{proof}
Let $(X,\|\cdot\|)$ be an $n$-dimensional normed space.
We show the claim by comparing the norm $\|\cdot\|$
with the Euclidean one $|\cdot|_E$ on $X$ provided by the John ellipsoid, satisfying
$\|x\| \le |x|_E \le n\|x\|$ for all $x \in X$.

Notice that the identity map $\id:(X,\|\cdot\|) \longrightarrow (X,|\cdot|_E)$ satisfies
the properties (\ref{StarRad}, \ref{StarAll}, \ref{StarGG}) of Lemma~\ref{LiftLemma}(\ref{Adv})
for $R = \infty$, $K_1 = n$, $K_2 = 1$ and the linear quasi-bicombings.
Take $\ee' > 0$ satisfying $\bb(\ee') \le \bb(\ee)/n$.
By Corollary~\ref{Hnk-ATB}(\ref{Hnk-ATB-Rn}),
there exists $L=L(n,\ee')$ such that every point in $(X, |\cdot|_E)$ satisfies
the $\ATB^*(\ee')$-condition with respect to the linear quasi-bicombing for $L$ and $R = \infty$.
Then, by Lemma~\ref{LiftLemma}(\ref{Adv}),
every point in $(X, \|\cdot\|)$ satisfies the $\ATB^*(\ee)$-condition
with respect to the linear quasi-bicombing for $L$ and $R = \infty$.
\end{proof}



Combining this with Theorem~\ref{ATB-to-SC} provides an alternative proof of the result in \cite{ST},
the rectifiability of bounded self-contracted curves in finite-dimensional normed spaces.

\subsection{Finsler manifolds}

From the study of the case of normed spaces in the previous subsection,
it is natural to expect a generalization to Finsler manifolds.
We refer to the books \cite{IntroRFG,ShenLFG} for the basics of Finsler geometry.
By a \emph{$C^{\infty}$-Finsler manifold}
we mean a connected $C^{\infty}$-manifold $M$ without boundary
equipped with a nonnegative function $F:TM \longrightarrow [0,\infty)$ satisfying:
\begin{enumerate}[(1)]
\item $F$ is $C^{\infty}$ on $TM \setminus 0$,
\item $F(cv)= \vert c \vert F(v)$ for all $v \in TM$ and $c \in \R$,
\item For any $x \in M$, the function $F^2|_{T_xM}$ has the positive-definite Hessian
on $T_xM \setminus \{0\}$.
\end{enumerate}

In general we assume the homogeneity $F(cv)=cF(v)$ only for $c>0$,
however, in this article we assume the \emph{reversibility} ($F(-v)=F(v)$)
in order to restrict ourselves to metric spaces in the ordinary sense for simplicity.

Associated with $F$ one can define the distance function
\[ d(x,y):=\inf_{\eta} \int_0^1 F\big( \dot{\eta}(t) \big) \,dt, \]
where $\eta:[0,1] \longrightarrow M$ runs over all piecewise $C^1$-curves from $x$ to $y$.
(The reversibility implies $d(y,x)=d(x,y)$.)
A $C^{\infty}$-curve $\gamma:I \longrightarrow M$ from an interval $I \subset \R$
is called a \emph{geodesic} if it is locally $d$-minimizing and
has constant speed ($F(\dot{\eta})$ is constant).
We say that $(M,F)$ is \emph{complete} if any geodesic $\gamma:[0,1] \longrightarrow M$
can be extended to a geodesic from $\R$.
In other words, the \emph{exponential map} $\exp$
(sending $v \in TM$ to $\gamma(1)$ with $\gamma$ being the geodesic with $\dot{\gamma}(0)=v$)
is well-defined on the whole tangent bundle.

\begin{proposition}[Local $\ATB$ of Finsler manifolds]\label{Fin-are-ATB}
Every Finsler manifold $(M,F)$ is locally $\ATB(\ee)$ for any $0<\ee<\pi/2$.
\end{proposition}

\begin{proof}
Fix arbitrary $x \in M$ and take small $R>0$ such that
any two points in $B^M_R(x)$ are connected by a unique minimizing geodesic and,
for any $p \in B^M_R(x)$, the exponential map
$\exp_p: B^{T_pM}_R(0) \longrightarrow B^M_R(p)$
is $C$-bi-Lipschitz for some $C \ge 1$ in the sense that
\[ C^{-1} F(v'-v) \le d\big( {\exp}_p(v),\exp_p(v') \big) \le CF(v'-v) \]
for all $v,v' \in B^{T_pM}_R(0)$. 
The uniqueness of minimizing geodesics ensures that
$B^M_R(p)$ has the unique quasi-bicombing.

The map $\exp_p^{-1}:B^M_R(p) \longrightarrow T_pM$ satisfies
the conditions in Lemma~\ref{LiftLemma}(\ref{Adv})
for $K_1=1$, $K_2 = C^{-1}$, and the unique quasi-bicombings.
Take $\ee'>0$ satisfying $\bb(\ee') \le \bb(\ee)/C$.
By Proposition~\ref{pr:norm}, $0 \in T_pM$ satisfies the $\ATB^*(\ee')$-condition
for $L = L_N(n, \ee')$ and $R=\infty$, where $n=\dim M$.
Hence Lemma~\ref{LiftLemma}(\ref{Adv}) implies that
$p$ satisfies the $\ATB(\ee)$-condition for $L$ and $R$.
\end{proof}


We can improve Proposition~\ref{Fin-are-ATB} into a global one
on Berwald spaces of nonnegative flag curvature, with the help of an argument in \cite{IL}.
A \emph{Berwald space} is a Finsler manifold whose covariant derivative
is independent of the choice of a reference vector,
and then all tangent spaces are mutually linearly isometric.
The \emph{flag curvature} is a generalization of the sectional curvature in Riemannian geometry.
See \cite{IntroRFG} or \cite{IL} for the definition and more properties of Berwald spaces.
For example, Riemannian manifolds and (Minkowski) normed spaces are Berwald spaces.

\begin{proposition}[$\ATB$ of Berwald spaces]\label{Ber-are-ATB}
For $n \in \N$ and $0<\ee<\pi/2$, there exists $L_B=L_B(n,\ee)$ such that
every complete $n$-dimensional Berward space $(M,F)$ of nonnegative flag curvature
is globally $\ATB(\ee)$ for $L_B$.
\end{proposition}

\begin{proof}
Fix $p \in M$, $y,z \in M \setminus \{p\}$ and let $\eta,\xi:[0,1] \longrightarrow M$
be minimal geodesics from $p$ to $y$, $p$ to $z$, respectively.
If there is no conjugate point to $p$,
then being a Berwald space of nonnegative curvature implies
\begin{equation}\label{eq:Buse}
d\big( \eta(t),\xi(t) \big) \ge t d\big( \eta(1),\xi(1) \big)
\end{equation}
for all $t \in [0,1]$
(this is the so-called \emph{Busemann concavity}, see \cite[Remark before Lemma~2.7]{Kell}).
In general, as in \cite{IL},
we can find a Riemannian metric $g$ which is affinely equivalent to $F$,
having the nonnegative sectional curvature, and satisfies
\[ C_n^{-1} F(v) \le \sqrt{g(v,v)} \le C_n F(v) \]
for all $v \in TM$ for some constant $C_n \ge 1$ depending only on $n$
(the Binet--Legendre metric in \cite{MT} does this job for example).
Since the Busemann concavity \eqref{eq:Buse} globally holds
for complete Riemannian manifolds of nonnegative sectional curvature,
we obtain
\begin{align*}
d_F\big(\eta(t),\xi(t) \big) &\ge C_n^{-1} d_g\big(\eta(t),\xi(t) \big)
 \ge C_n^{-1} t d_g\big(\eta(1),\xi(1) \big) \\
&\ge C_n^{-2} t d_F\big(\eta(1),\xi(1) \big).
\end{align*}
By taking the limit of $d_F(\eta(t),\xi(t))/t$ as $t \to 0$,
we have $d_F(\eta(1),\xi(1)) \le C_n^2 F(\dot{\xi}(0)-\dot{\eta}(0))$.
Hence one can finish the proof by applying Lemma~\ref{LiftLemma}(\ref{Adv})
in the same way as in Proposition~\ref{Fin-are-ATB} with $R=\infty$.
\end{proof}

\subsection{Cayley graphs of virtually abelian groups}\label{Cayley}

In this subsection we prove the following.

\begin{proposition}[$\ATB$ of Cayley graphs]\label{virt-abel}
Let $G$ be a finitely generated
  virtually abelian group with a given symmetric set
  of generators, and $K$ be the Cayley graph of
  $G$ with respect to this set of generators.
  Then $K$ is globally $\ATB(\ee)$ for every $\ee>0$.
\end{proposition}

Being \emph{virtually abelian} means that $G$ has an abelian subgroup of finite index.
The proof of Proposition~\ref{virt-abel} is based on the observation that 
the Cayley graph of a virtually abelian group is 
of finite Gromov--Hausdorff distance from a normed space.
This observation follows from Theorem~\ref{group-norm} below,
which generalizes a result by D.~Burago \cite{BU}
(see also \cite{MO, BIK})
for the universal covering of a torus.
The proof mainly follows the same lines.

\subsubsection{Periodic metrics }

\begin{theorem}\label{group-norm}
Let $(X,d)$ be a proper, geodesic metric space
on which $\Gamma=\ZZ^n$ acts
cocompactly and properly discontinuously by isometries.
Suppose that there is a continuous map
$F:X \longrightarrow \RR^n$ such that the action $\Gamma \curvearrowright X$
is equivariant to the standard action
$\ZZ^n \curvearrowright \RR^n$ by shifts.
Then there exists a norm $\Vert \cdot \Vert$ of $\RR^n$
such that the Gromov--Hausdorff distance between $(X,d)$ and $(\RR^n,\Vert \cdot \Vert)$
is finite.
\end{theorem}

\begin{proof}
Fix $x_0\in X$ and assume $F(x_0)=0$ without loss of generality.
Define the function
\[ \Vert g\Vert := \lim_{k\to\infty}\frac{d(x_0,g^k(x_0))}{k}. \]
This limit exists by Fekete's sub-additive lemma
because the function $k \longmapsto d(x_0, g^k(x_0))$ is sub-additive,
and the limit does not depend on the choice  of $x_0$.
Note that the function $\Vert\cdot\Vert$ satisfies the homogeneity
\begin{equation}\label{homog}
\Vert g^m \Vert =|m| \cdot \Vert g \Vert
\end{equation}
for $m \in \ZZ$ as well as the triangle inequality
\begin{equation}\label{triang}
\Vert g\cdot h\Vert\le\Vert g\Vert+\Vert h\Vert
\end{equation}
for $g,h \in \Gamma$ (see, e.g., \cite{LebNorm}).

We need two more lemmas to finish the proof of the theorem.
The first one is an easy observation on sub-additive functions and the proof is omitted.

\begin{lemma}\label{subad} 
Let $f:\ZZ \longrightarrow (0,\infty)$ be a sub-additive function
such that $\lim_{n \to \infty} f(n) = \infty$ and 
$f(n)\le f(2n)/2+C$ for some $C>0$ and all $n \in \ZZ$. 
Then we have $c:=\lim_{n\to\infty} f(n)/n>0$ and
$cn\le f(n)\le cn+2C$ for all $n \in \ZZ$.
\end{lemma}

The next lemma is the main technical ingredient of the proof of Theorem~\ref{group-norm},
that allows us to apply Lemma~\ref{subad} to the function $d(x_0, g^k(x_0))$.

\begin{lemma}\label{Keylemma}
There exists $C\ge 0$ such that $d(x_0, g(x_0))\le d(x_0,g^2(x_0))/2+C$
holds for every $g \in \Gamma$.
\end{lemma}

Let us postpone the proof of Lemma~\ref{Keylemma} and finish the proof of the theorem.
It follows from Lemmas~\ref{subad} and \ref{Keylemma} that
$\Vert\cdot\Vert$ is bounded away from $0$.
Hence, together with \eqref{homog} and \eqref{triang},
we find that $\Vert \cdot \Vert$ defines a norm of $\R^n$
including $\ZZ^n=\Gamma$.
It also follows from Lemma~\ref{subad} that $(\Gamma(x_0),d)$
is of finite Gromov--Hausdorff distance from $(\Gamma, \Vert\cdot\Vert)$.
Since $\Gamma$ acts cocompactly on $X$,
Theorem~\ref{group-norm} follows.
\end{proof}

It remains to prove Lemma~\ref{Keylemma}.

\begin{proof}[Proof of  Lemma~$\ref{Keylemma}$]
The proof is based on the following analogue of the
intermediate value theorem proven in \cite{BU}.

\begin{lemma}\label{intermed}
Let $f:[0,1] \longrightarrow \RR^n $ be a continuous map with $f(1)=0$.
Then there exist $0\le a_1\le a_2\le\cdots\le a_n\le 1$ such that
\[ \sum_{i=1}^n(-1)^{i+1}f(a_i)=\frac{f(0)}{2}. \]
In other words, if we set $a_0 = 0$ and $a_{n+1} = 1$, then we have
\begin{equation}
\sum_{\substack{0 \le i \le n, \\ i \equiv 1\bmod 2}} \big( f(a_i) - f(a_{i+1}) \big) = 
\sum_{\substack{0 \le i \le n, \\ i \equiv 0\bmod 2}} \big( f(a_i) - f(a_{i+1}) \big) = 
\frac{f(0)}{2}.\label{2Copies}
\end{equation}
\end{lemma}


The next lemma is straightforward.

\begin{lemma}\label{preimage}
For any $r>0$ there exists $R \ge 0$ such that, 
for every $x_1, x_2 \in X$ with $\vert F(x_1) - F(x_2) \vert \le r$, 
we have $d(x_1, x_2) \le R$. 
We will denote the minimum of such $R$ by $\Theta(r)$.
\end{lemma}


Fix $g \in \GG$.
The plan of the proof is as follows.
We will find a path $\JJ:[0, 1] \longrightarrow X$
connecting $g^2(x_0)$ and $x_0$ satisfying the following.

\begin{enumerate}
\item{$\JJ$ is an almost shortest path in the sense that 
\begin{equation}
L(\JJ) \le d \big( x_0, g^2(x_0) \big) + 2n\TTT(\sqrt{n}). \label{preEq1}
\end{equation}
}
\item{$\JJ$ is almost containing two copies of a shortest path between $g(x_0)$ and $x_0$
in a certain sense, and we have
\begin{equation}
L(\JJ) \ge 2d\big( x_0,g(x_0) \big) - 2\TTT(n \sqrt{n}). \label{preEq2}
\end{equation}
}
\end{enumerate}
Notice that Lemma~\ref{Keylemma} readily follows from (\ref{preEq1}) and (\ref{preEq2}).

\textit{Step~$1:$ The construction of $\JJ$.} 
Let $\gamma:[0,1] \longrightarrow X$ be a minimizing geodesic from $g^2(x_0)$ to $x_0$.
We apply Lemma~\ref{intermed} to the map $\wt\gamma=F\circ\gamma$ 
(recall $F(x_0)=0$) and obtain 
$0 = a_0 \le a_1 \le a_2 \le \dots \le a_n \le a_{n+1} = 1$ such that

\begin{equation}
\sum_{\substack{0 \le i \le n, \\ i \equiv 1 \bmod 2}} \big( \wt\gg(a_i) - \wt\gg(a_{i+1}) \big) = 
\sum_{\substack{0 \le i \le n, \\ i \equiv 0 \bmod 2}} \big( \wt\gg(a_i) - \wt\gg(a_{i+1}) \big) = 
\frac{\wt\gg(0)}{2}.\label{halb}\end{equation}


Let $[\cdot]:\RR^n \longrightarrow \ZZ^n$ be the coordinate floor function,
namely $[x]_j \le x_j < [x]_j +1$.
We fix some $p_i \in F^{-1}([\wt\gamma(a_i)])$ for $1 \le i \le n$,
and also set $p_0 := g^2(x_0)$ and $p_{n+1} := x_0$.
Now we define $\Pi:[0,1] \longrightarrow X$ as the concatenation of the shortest paths $p_ip_{i+1}$, 
where $\Pi(a_i)=p_i$.

\textit{Step~$2:$ The proof of \eqref{preEq1}.} 
We deduce from $\vert [\wt\gamma(a_i)] - \wt\gamma(a_i) \vert \le \sqrt{n}$
and Lemma~\ref{preimage} that $d(p_i, \gamma(a_i)) \le \TTT(\sqrt{n})$. 
Then the triangle inequality implies \eqref{preEq1} as
\begin{align*}
L(\Pi) &= \sum_{i=0}^n d(p_i,p_{i+1})
 \le \sum_{i=0}^n d\big( \gamma(a_i),\gamma(a_{i+1}) \big) + 2n \TTT(\sqrt{n}) \\
&= d \big( x_0, g^2(x_0) \big) + 2n \TTT(\sqrt{n}),
\end{align*}
where we have $2n$ instead of $2(n+1)$ since $p_0=g^2(x_0)=\gamma(0)$
and $p_{n+1}=x_0=\gamma(1)$.

\textit{Step~$3:$ The proof of \eqref{preEq2}.} 
For $0 \le i \le n$, we denote by $g_i$ an element of $\GG$ such that $g_i(p_{i + 1}) = p_{i}$. 
Then we can write 
\[ L(\Pi) = \sum_{i=0}^n d(p_i, p_{i+1})
 = \sum_{i=0}^n d\big( g_i(p_{i+1}), p_{i+1} \big). \]
Notice that, if $x$ and $y$ are in the same orbit of an isometric action of an abelian group,
then for any element $\aa$ of this group we have $d(x, \aa(x)) = d(y, \aa(y))$.
Thus we observe
\begin{align*}
L(\Pi) &= \sum_{\substack{0 \le i \le n, \\ i \equiv 0 \bmod 2}} d\big( g_i(p_{i+1}), p_{i+1} \big)
 + \sum_{\substack{0 \le i \le n, \\ i \equiv 1 \bmod 2}} d\big( g_i(p_{i+1}), p_{i+1} \big) \\
&= \sum_{\substack{0 \le i \le n, \\ i \equiv 0 \bmod 2}}
 d\bigg(g_i \Big(  \Big(\prod_{\substack{i < k \le n, \\ k \equiv 0 \bmod 2}}g_k\Big)(x_0) \Big), 
 \Big(\prod_{\substack{i < k \le n, \\ k \equiv 0 \bmod 2}}g_k\Big)(x_0)\bigg) \\
&\quad + \sum_{\substack{0 \le i \le n, \\ i \equiv 1 \bmod 2}}
 d\bigg(g_i \Big(  \Big(\prod_{\substack{i < k \le n, \\ k \equiv 1 \bmod 2}}g_k\Big)(x_0) \Big), 
 \Big(\prod_{\substack{i < k \le n, \\ k \equiv 1 \bmod 2}}g_k\Big)(x_0)\bigg).
\end{align*}
Applying the triangle inequality to each of those sums provides 
\[ L(\Pi) \ge
 d\bigg( \Big(\prod_{\substack{0 \le k \le n, \\ k \equiv 0 \bmod 2}}g_k\Big)(x_0), x_0 \bigg) +
d\bigg( \Big(\prod_{\substack{0 \le k \le n, \\ k \equiv 1 \bmod 2}}g_k\Big)(x_0), x_0 \bigg). \]
Now, again by the triangle inequality,
to provide \eqref{preEq2} it suffices to show that for $m = 1,2$
\[ d\bigg( \Big(\prod_{\substack{0 \le k \le n, \\ k \equiv m \bmod 2}}g_k\Big)(x_0), g(x_0) \bigg)
 \le \TTT(n \sqrt{n}). \]
Thanks to Lemma~\ref{preimage}, the above inequality follows from
\begin{equation}\label{F-F}
\bigg \vert F\bigg(\Big(\prod_{\substack{0 \le k \le n, \\ k \equiv m \bmod 2}}g_k\Big)(x_0)\bigg)
 -  F\big( g(x_0) \big) \bigg \vert \le n \sqrt{n}.
\end{equation}
Let $\wt g_k \in \ZZ^n$ (resp.\ $\wt g$) the element corresponding to $g_k$ (resp.\ $g$).
Then, since $F$ is $\Gamma$-equivariant, we have
\[ F\bigg(\Big(\prod_{\substack{0 \le k \le n, \\ k \equiv m \bmod 2}}g_k\Big)(x_0)\bigg)
 -  F\big( g(x_0) \big)
 =\sum_{\substack{0 \le k \le n, \\ k \equiv m \bmod 2}} \wt g_k -\wt g. \]
Therefore \eqref{F-F} is rewritten as
\[ \bigg \vert \sum_{\substack{0 \le k \le n, \\ k \equiv m \bmod 2}} \big( [\wt\gg(a_{k})] - [\wt\gg(a_{k+1})] \big)
 - \frac{\wt\gg(0)}{2} \bigg \vert \le n \sqrt{n}, \]
which follows from \eqref{halb}, $\vert [\wt\gamma(a_i)] - \wt\gamma(a_i) \vert \le \sqrt{n}$
for $1 \le i \le n$ and $\wt\gamma(a_i)=[\wt\gamma(a_i)]$ for $i=0,n+1$.
This completes the proof of Lemma~\ref{Keylemma} and then Theorem~\ref{group-norm}.
\end{proof}

Applying Theorem~\ref{group-norm} to the setting of Proposition~\ref{virt-abel}
yields the following.

\begin{corollary}\label{corGHCayley}
Let $G$ be a finitely generated
  virtually abelian  group with a given symmetric set
  of generators and $(K,\rho)$ be the Cayley graph of
  $G$ with respect to this set of generators.
  Then there exists a norm $\Vert \cdot \Vert$ on $\RR^n$
such that the Gromov--Hausdorff distance between $(K,\rho)$ and $(\RR^n,\Vert \cdot \Vert)$
is finite.
\end{corollary}

\begin{proof}
Let 
 $\Gamma\sim\ZZ^n$ be a subgroup of finite index in $G$
 (we have $n<\infty$ because any subgroup of finite index
 in a finitely generated group is finitely generated).
 Then $\Gamma$ acts on $K$ cocompactly and properly
 discontinuously. Let us construct a map $F:K \longrightarrow \RR^n$
 satisfying the conditions of
Theorem~\ref{group-norm} (without using the Bieberbach theorem). 
Let $x_0\in K$ correspond to $e\in G$.
We fix an isomorphism $g \longmapsto \wt g$ between
 $\Gamma$ and $\ZZ^n$ and map the orbit $\Gamma(x_0)$
 to the integer lattice via this isomorphism. 
 Then choose one representative from every right coset with respect to $\Gamma$,
 say $g_1,\dots, g_k$, and map 
$g_1(x_0),\dots, g_k(x_0)$ to $\RR^n$ in an arbitrary way.
Then define the map on $G(x_0)$ by 
$F(hg_i(x_0)):=F(g_i(x_0)) +\wt h$ for $h \in \Gamma$.
Finally extend $F$ linearly to the edges.
Then we apply Theorem~\ref{group-norm}
to obtain that $K$ is of finite distance from some 
$n$-dimensional normed space.
\end{proof}

\subsubsection{Proof of Proposition~$\ref{virt-abel}$}

For $a,b,c >0$ satisfying the triangle inequalities
\begin{equation}\label{abc}
a \le b+c, \qquad b \le c+a, \qquad c \le a+b,
\end{equation}
let us denote by $\angle(a,b; c)$
the angle between $a$ and $b$ for the flat triangle with edges $a, b, c$,
namely $c^2=a^2 +b^2 -2ab\cos\angle(a,b;c)$.
By an elementary calculation we observe the following.

\begin{lemma}\label{tre}
For every $d,\delta>0$, there exists $D=D(d,\delta)>0$ such that, 
for any triples $a,b,c>0$ and $a', b', c'>0$ satisfying \eqref{abc},
$a,b>D$, $c>0$, $|a-a'| \le d$, $|b-b'| \le d$ and $|c-c'| \le d$,
we have $|\angle(a',b'; c')-\angle(a,b; c)|\le\delta$.
\end{lemma}

The next lemma is crucial in the proof of Proposition~\ref{virt-abel}.

\begin{lemma}\label{loc-glob}
Let $(X,d_X)$ be a metric space and $\ee>0$ such that, for every $R>0$,
there is $L_X=L_X(R)>0$ for which every point $p \in X$ satisfies $\ATB(\ee)$
with constants $L_X$ and $R$.
Suppose further that $X$ is of finite Gromov--Hausdorff distance from
some metric space $(Y,d_Y)$ which is globally $\ATB(\ee/2)$.
Then $X$ is globally $\ATB(\ee)$.
\end{lemma}

\begin{proof}
Denote the Gromov--Hausdorff distance between $X$ and $Y$ by $d/2$,
and take $D=D(d,\ee/2)$ from Lemma~\ref{tre}.
Let then $L_X=L_X(D)$ be the constant provided by the hypothesis on $X$,
and $Y$ be globally $\ATB(\ee/2)$ with $L_Y$.
We shall show that $X$ is globally $\ATB(\ee)$ with the constant $L_X+L_Y-1$.
To this end, given $p\in X$,
let $y_1, y_2,\dots, y_N\in X\setminus\{p\}$ satisfy $\wt\angle y_ipy_j\ge\ee$ for all $i\neq j$.
Then, on the one hand,
we immediately find $\#[\{y_1, y_2,\dots, y_N \} \cap B_D(p)] <L_X$.
On the other hand, we obtain
$\#[\{y_1, y_2,\dots, y_N \} \setminus B_D(p)] <L_Y$ 
since, otherwise, the choice of $d$ and Lemma~\ref{tre} provide
a point $p'\in Y$ and $L_Y$ points $y'_i \in Y \setminus \{p'\}$
such that $\wt\angle y'_i p' y'_j\ge\ee/2$ for all $i \neq j$
and we have a contradiction with the choice of $L_Y$.
Therefore $N \le L_X +L_Y -2$ and this completes the proof.
\end{proof}

Now we are ready to prove Proposition~\ref{virt-abel}.
Note that the Cayley graph $K$ satisfies the first condition of Lemma~\ref{loc-glob}.
Then Corollary~\ref{corGHCayley} and Lemma~\ref{pr:norm} ensure that
Lemma~\ref{loc-glob} is indeed available with $Y$ a finite-dimensional normed space,
and we obtain Proposition~\ref{virt-abel}.

\subsection{Submetries and Quotient spaces}\label{submetry}

We finally discuss the behavior of the $\ATB$-condition
under some deformations of metric spaces, related to Lemma~\ref{LiftLemma}.

\begin{definition}[Submetries]
Let $(X,d_X)$ and $(Y,d_Y)$ be metric spaces.
A map $\psi: X \longrightarrow Y$ is called a \emph{submetry} if, for every $x \in X$ and $r > 0$,
the image of the closed ball $\overline{B_r(x)}$ coincides with $\overline{B_r(\psi(x))}$.
\end{definition}

Notice that every submetry is $1$-Lipschitz.

\begin{proposition}\label{SubATB}
Let $\psi:X \longrightarrow Y$ be a submetry and $0 < \ee < \pi/2$.
Suppose that $p \in X$ satisfies the $\ATB(\ee)$-condition for some $L \in \N$ and $R > 0$.
Then $q := \psi(p) \in Y$ also satisfies the $\ATB(\ee)$-condition for $L$ and $R$.  
\end{proposition}

\begin{proof}


It follows from the definition of the submetry that, for every $y \in Y$,
we have $S_{d_Y(q, y)}(p) \cap \psi^{-1}(y) \neq \emptyset$.
Thus we have (by the axiom of choice) $\Phi:Y \longrightarrow X$ such that 
$\Phi(y) \in S_{d_Y(q, y)}(p) \cap \psi^{-1}(y)$ and $\Phi(q)=p$. 
Since $y=\psi(\Phi(y))$ and $\psi$ is $1$-Lipschitz,
we conclude that $\Phi$ satisfies the conditions of Lemma~\ref{LiftLemma}(\ref{Simp}).
This completes the proof.
 \end{proof}

For a metric space $(X,d_X)$ with an isometric group action $G \curvearrowright X$ with closed orbits,
the quotient metric space $(X/G,d_{X/G})$ is defined in the following way.
Denoted by $X/G$ is the set of orbits, and the metric is given by 
\[ d_{X/G}(\wb x, \wb y) := \inf_{x \in \wb x,\, y \in \wb y} d_X(x,y). \]
Clearly, the quotient map $\pi: X \longrightarrow X/G$ is a submetry. 
Thus, we have the following corollary of Proposition~\ref{SubATB}. 

\begin{corollary}[$\ATB$ of Quotient spaces]
Let $(X,d_X)$ satisfy the local, semi-global, or global $\ATB(\ee)$-condition
and $G \curvearrowright X$ be an isometric group action with closed orbits.
Then $X/G$ also satisfies the local, semi-global, or global $\ATB(\ee)$-condition, respectively.
\end{corollary}

\section{Gradient curves of quasi-convex functions are self-contracted}\label{QCtoSC}

As we mentioned in the introduction,
self-contracted curves arise as gradient curves of quasi-convex functions.

\begin{definition}[Quasi-convexity]\label{quasi-convex}
A function $f:X \longrightarrow \R$ on a geodesic metric space $(X,d)$
is said to be \emph{quasi-convex} if,
for any $x,y \in X$ and any minimizing geodesic $\gg: [0,1] \longrightarrow X$
from $x$ to $y$, we have for all $t \in (0,1)$
\[ f\big( \gg(t) \big) \le \max\{f(x),f(y)\}. \]
\end{definition}

An equivalent definition is that every sub-level set $\{ x \in X \,|\, f(x) \le a \}$ is convex.
It was known that gradient curves of quasi-convex functions are self-contracted in
\begin{itemize}
\item{Euclidean spaces (see \cite[Proposition 6.2]{DLS}),}
\item{$\CAT(0)$-spaces (see \cite[Proposition 4.6]{OhRect}).}
\end{itemize}
In the latter case, to be precise, we assumed an auxiliary condition that
$f$ is either bounded below or \emph{$\lambda$-convex} for some $\lambda \in \R$ in the sense that
\[ f\big( \gamma(t) \big) \le (1-t)f(x) +tf(y) -\frac{\lambda}{2}(1-t)t d(x,y)^2 \]
for all $t \in (0,1)$ along every geodesic $\gamma:[0,1] \longrightarrow X$ from $x$ to $y$.

In this section we shall add spaces with lower curvature bound into the above list.
Our proof is based on the evolution variational inequality
(in the same spirit as \cite{OhRect})
and applies to other cases where the evolution variational inequality is known.
We remark that normed spaces and Finsler manifolds do not belong to this class
and the self-contractedness may fail
(see \cite{OS} for the failure of the contraction property
$d(\gg(t),\eta(t)) \le e^{-\lambda t}d(\gg(0),\eta(0))$ for pairs of gradient curves).

Given a function $f:X \longrightarrow \R$ on a metric space $(X,d)$,
an absolutely continuous curve $\gamma:[0,\infty) \longrightarrow X$ is said to satisfy
the \emph{$\lambda$-evolution variational inequality} ($\EVI_{\lambda}$ for short) if,
for any $y \in X$ and almost every $t>0$,
\begin{equation}\label{eq:EVI}
\frac{d}{dt} \bigg[ \frac{1}{2}d\big( \gamma(t),y \big)^2 \bigg]
 +\lambda d\big( \gamma(t),y \big)^2 +f\big( \gamma(t) \big)
 \le f(y).
\end{equation}
The existence of curves satisfying the EVI${}_{\lambda}$
characterizes the $\lambda$-convexity of $f$,
and then those curves turn out gradient curves of $f$ in the metric sense
(called the \emph{energy dissipation}) that means
\[ f\big( \gamma(t) \big) =f\big( \gamma(s) \big)
 -\frac{1}{2} \int_s^t \big\{ |\dot{\gamma}(r)|^2 +|\nabla f|\big( \gamma(r) \big)^2 \big\} \,dr \]
for all $s<t$, where $|\dot{\gamma}|$ is the metric speed of $\gamma$
and $|\nabla f|$ is the (descending) slope of $f$.
The inequality \eqref{eq:EVI} possesses an overwhelming importance
in gradient flow theory on ``Riemannian-like'' spaces,
we refer to the book \cite{AGSbook} for more on gradient flow theory on metric spaces.

Although our main interest is in metric spaces with lower sectional curvature bounds,
here we discuss more general Ricci curvature bounds.
The \emph{Riemannian curvature-dimension condition} $\mathrm{RCD}(K,\infty)$
is a synthetic notion, for metric measure spaces, of the lower Ricci curvature bound
defined as the combination of the $K$-convexity of the relative entropy on the $L^2$-Wasserstein space
and the linearity of the corresponding heat flow (see \cite{St-I,LV,AGSrcd}).
It is known by Sturm \cite{Sturm} that $\lambda$-convex functions
on locally compact $\mathrm{RCD}(K,\infty)$-spaces satisfy $\EVI_{\lambda}$ \eqref{eq:EVI}.
This is the starting point of our proof of the self-contractedness.

In \cite{Petrunin} (see also \cite{ZZ}) it was shown that
any $n$-dimensional Alexandrov space of curvature $\ge k$
equipped with the $n$-dimensional Hausdorff measure
satisfies $\mathrm{RCD}(K,\infty)$ for $K=(n-1)k$.
Therefore the following proposition covers the situation treated in
Corollaries~\ref{BBC-to-SC}, \ref{BBC-no-SF}.
We also refer to \cite{Lytchak,Savare,OhtaGF} for direct studies of gradient curves
in Alexandrov spaces as well as in the Wasserstein spaces over them.

\begin{proposition}[Quasi-convexity of gradient curves]\label{GC-is-QC}
Let $(X,d,m)$ be a locally compact metric space equipped with a Borel measure $m$,
and assume that it satisfies $\mathrm{RCD}(K,\infty)$ for some $K \in \R$.
Then, for any lower semi-continuous, quasi-convex function $f:X \longrightarrow \R$
which is also $\lambda$-convex for some $\lambda \in \R$,
every gradient curve of $f$ is self-contracted.
\end{proposition}

\begin{proof}
Let $\gamma:[0,\ell) \longrightarrow X$ be a gradient curve of $f$ and take $T \in (0,\ell)$.
Fix arbitrary $t_0 \in (0,T)$ and let $\xi:[0,1] \longrightarrow X$ be a minimal geodesic
from $\gamma(t_0)$ to $\gamma(T)$ (the existence follows from $\mathrm{RCD}(K,\infty)$).
For each $s \in (0,1)$, we observe from the triangle inequality that
\begin{align*}
&\limsup_{t \to t_0} \frac{d(\gamma(t),\gamma(T))^2 -d(\gamma(t_0),\gamma(T))^2}{2(t-t_0)} \\
&\le d\big( \gamma(t_0),\gamma(T) \big) 
 \limsup_{t \to t_0} \frac{d(\gamma(t),\xi(s)) -d(\gamma(t_0),\xi(s))}{t-t_0} \\
&= \frac{d(\gamma(t_0),\gamma(T))}{d(\gamma(t_0),\xi(s))}
 \limsup_{t \to t_0} \frac{d(\gamma(t),\xi(s))^2 -d(\gamma(t_0),\xi(s))^2}{2(t-t_0)}.
\end{align*}
Thanks to (the integration of) \eqref{eq:EVI} with $y=\xi(s)$, we have
\begin{align*}
&\limsup_{t \to t_0} \frac{d(\gamma(t),\xi(s))^2 -d(\gamma(t_0),\xi(s))^2}{2(t-t_0)}
 +\lambda d\big( \gamma(t_0),\xi(s) \big)^2 \\
&\le f\big( \xi(s) \big) -f\big( \gamma(t_0) \big) \le 0,
\end{align*}
where the second inequality follows from the quasi-convexity of $f$
since $f(\gamma(T)) \le f(\gamma(t_0))$ holds for $\gg$ being a gradient curve.
Letting $s \to 0$, we obtain
\[ \limsup_{t \to t_0} \frac{d(\gamma(t),\gamma(T))^2 -d(\gamma(t_0),\gamma(T))^2}{2(t-t_0)}
 \le 0 \]
for all $t_0 \in (0,T)$. 
Hence $d(\gamma(t),\gamma(T))$ is non-increasing in $t \in [0,T]$.
\end{proof}

Let us finally remark that the self-contractedness of gradient curves
does not characterize quasi-convexity, consider for instance $f(x)=\sin x$ on $\R$.
We need a ``global'' property, such as the EVI (between a curve and a point)
or the contraction property (between two curves), for such a characterization of the global convexity.

In a different approach, nevertheless, it was shown in a recent paper \cite{DuCaLe}
that a $C^{1,\alpha}$-curve ($\alpha \in (1/2,1]$) in $\R^n$ 
satisfies a strict version of the self-contractedenss (called the \emph{strong self-contractedness})
if and only if its reparametrization is realized as a gradient curve of a $C^1$-convex function.

\section{Doubling condition and absence of large $\SRA(\aa)$-subsets}\label{sc:Doubling}

In this section we show that the absence of large $\SRA(\aa)$-subsets
implies the doubling condition.
This could be regarded as a weak converse to Theorem~\ref{ATB-no-SRA},
and gives another interesting connection
between the rectifiability of bounded self-contracted curves and the finiteness of dimension.
A metric space $(X,d)$ is said to satisfy the (metric) \emph{doubling condition} if
there exists $L \in \N$ such that, for any $x \in X$ and $R>0$,
we can find $x_1,\ldots,x_L \in B_R(x)$ satisfying
$B_R(x) \subset \bigcup_{i=1}^L B_{R/2}(x_i)$.

\begin{theorem}\label{MeetDoubling}
Let $(X,d)$ be a metric space, and suppose that there exist $0 < \aa < 1$ and $N \in \N$
such that $X$ does not admit any $\SRA(\aa)$-subset of cardinality $\ge N$.
Then $X$ satisfies the doubling condition with $L=L(N,\alpha)$.
\end{theorem}

\begin{proof} Let $\wt N \in \N$  be such that $\aa(\wt N-2) \ge 3$,
and we consider the Ramsey number $H = \RRR(\wt N, \wt N, \wt N ,\wt N , \wt N, \wt N, N; 3) \in \N$.
This means that, if we take an $H$-element set $\M$,
then for every coloring of $3$-point subsets of $\M$ in $7$ colors,
we have the following property:
Either there exists an $N$-element subset of $\M$ such that
all of its $3$-point subsets are of the seventh color,
or there exists an $\wt N$-element subset of $\M$ such that
all of its $3$-point subsets are of the same color from the first to the sixth.  

We shall show that the doubling condition holds with $L=H-1$.
Assume in contrary that there exist $x \in X$ and $R > 0$
such that the open $(R/2)$-neighborhood of any $L$-point set in $B_R(x)$ cannot cover $B_R(x)$.
Then one can inductively choose $x_1,\dots,x_{L+1} \in B_R(x)$
with $d(x_i,x_j) \ge R/2$ for every $i \neq j$.
We color triples of points in this set in $7$ colors.
For $i < j < k$ we color $(x_i,x_j,x_k)$ by the following rule:
\begin{enumerate}
\item{$d(x_i,x_k)  \ge d(x_i,x_j) + \aa d(x_j,x_k)$,}\label{c1}
\item{not (\ref{c1}) but $d(x_j,x_k) \ge d(x_i,x_j) + \aa d(x_i,x_k)$,}\label{c2}
\item{not (\ref{c1})--(\ref{c2}) but $d(x_j,x_k) \ge d(x_i,x_k) + \aa d(x_i,x_j)$,}\label{c3}
\item{not (\ref{c1})--(\ref{c3}) but $d(x_i,x_k) \ge d(x_j,x_k) + \aa d(x_i,x_j)$,}\label{c4}
\item{not (\ref{c1})--(\ref{c4}) but $d(x_i,x_j) \ge d(x_j,x_k) + \aa d(x_i,x_k)$,}\label{c5}
\item{not (\ref{c1})--(\ref{c5}) but $d(x_i,x_j) \ge d(x_i,x_k) + \aa d(x_j,x_k)$,}\label{c6}
\item{none of the above.}\label{c7}
\end{enumerate}
By the definition of $L+1=H$, either there exists an $N$-point set satisfying (\ref{c7}) or 
there exists an $\wt N$-point set satisfying one of the properties (\ref{c1})--(\ref{c6}).
The first case contradicts the assumption on the absence of an $N$-point $\SRA(\aa)$-set. 
Now we are going to deal with the second case.
By renumbering, let $x_1,\ldots,x_{\wt N}$ be the $\wt N$ points in question.

\textit{Case}~1: $x_1,\dots,x_{\wt N}$  are of the first color. 
Taking the sum of the inequalities
\begin{align*}
d(x_1,x_3) &\ge d(x_1,x_2) + \aa d(x_2,x_3), \\
d(x_1,x_4) &\ge d(x_1,x_3) + \aa d(x_3,x_4), \\
 &\ldots, \\
d(x_1,x_{\wt N}) &\ge d(x_1,x_{\wt{N}-1}) + \aa d(x_{\wt N-1},x_{\wt N}),
\end{align*}
we have, together with the choice of $\wt N$,
\begin{align*}
d(x_1,x_{\wt N}) &\ge d(x_1,x_2)
 + \aa \big( d(x_2,x_3) + \cdots + d(x_{\wt N-1},x_{\wt N}) \big) \\
&\ge \frac{R}{2} + \aa(\wt N-2)\frac{R}{2} \ge 2R.
\end{align*}
This contradicts the assumption that $x_1,x_{\wt N}  \in B_R(x)$.

\textit{Case}~2: $x_1,\dots,x_{\wt N}$ are of the second color. 
We have the inequalities  
\begin{align*}
d(x_2,x_3) &\ge d(x_1,x_2) + \aa d(x_1,x_3), \\
d(x_3,x_4) &\ge d(x_2,x_3) + \aa d(x_2,x_4), \\
 &\ldots, \\
d(x_{\wt N-1},x_{\wt N}) &\ge d(x_{\wt N-2},x_{\wt N-1}) + \aa d(x_{\wt N-2},x_{\wt N}).
\end{align*}
Summing them provides, similarly to the previous case,
\[ d(x_{\wt N-1},x_{\wt N}) \ge d(x_1,x_2)
 + \aa \big( d(x_1,x_3) + \cdots + d(x_{\wt N-2},x_{\wt N}) \big) \ge 2R, \]
which is a contradiction.

\textit{Case}~3: $x_1,\dots,x_{\wt N}$ are of the third color. 
In this case we sum
\begin{align*}
d(x_2,x_{\wt N}) &\ge d(x_1,x_{\wt N}) + \aa d(x_1,x_2), \\
d(x_3,x_{\wt N}) &\ge d(x_2,x_{\wt N}) + \aa d(x_2,x_3), \\
 &\ldots, \\
d(x_{\wt N - 1},x_{\wt N}) &\ge d(x_{\wt N - 2},x_{\wt N}) + \aa d(x_{\wt N - 2},x_{\wt N - 1})
\end{align*}
to obtain
\[ d(x_{\wt N - 1},x_{\wt N}) \ge d(x_1,x_{\wt N})
 + \aa \big( d(x_1,x_2) + \cdots + d(x_{\wt N - 2},x_{\wt N - 1}) \big) \ge 2R. \]
This is again a contradiction.
The cases produced by the forth, fifth and sixth inequalities
can be reduced to the first, second and third cases, respectively,
by reversing the order of $x_1,\ldots,x_{\wt N}$.
This completes the proof.
\end{proof}

\section{Counter-examples}\label{Counter}

This section is devoted to several kinds of examples of metric spaces
including infinite (or arbitrarily large) $\SRA(\aa)$-subsets
or unrectifiable bounded self-contracted curves.
Some of them satisfy the doubling condition,
therefore the converse of Theorem~\ref{MeetDoubling} does not hold in general.

\subsection{Heisenberg group}\label{Heisenberg}

As a first example of particular interest,
let us consider the \emph{Heisenberg group} $(\R^3,\mathcal{H})$
(we refer to \cite{Mont}).
Let $\gamma:[0,1] \longrightarrow \R^3$ be defined by $\gamma(t)=(0,0,t)$.
The distance between $(0,0,s)$ and $(0,0,t)$ is $2\sqrt{\pi|s-t|}$,
thus the $z$-axis is isometric to the $(1/2)$-snowflake of $\R$.
Hence $\gamma$ is bounded and self-contracted,
whereas every small piece of $\gamma$ has infinite length.

On the one hand,
the Heisenberg group is known to satisfy the \emph{measure contraction property}
(in the sense of \cite{O-mcp,St-II}) and, in particular, the doubling condition.
On the other hand, it does not satisfy the curvature-dimension condition $\CD(K,N)$
for any parameters $K,N$ (see \cite{Juillet}).

\subsection{Metric trees}

For $x,y \in \R^2$, we denote by $[x,y] \subset \R^2$ the segment connecting $x$ and $y$.
Let $\{t_i\}_{i = 1}^{\infty} \subset (0,1]$ be a strictly decreasing sequence
such that $\lim_{i \to \infty} t_i = 0$,
and define $\BB \subset \R^2$ by 
\[ \BB = [(0,0), (1,0)] \cup \bigcup_{i = 1}^{\infty}[(t_i, 0),(t_i, t_i)]. \]
We consider $\BB$ as a metric space with the intrinsic distance denoted by $d_{\BB}$.
Note that $(\BB, d_{\BB})$ is a compact metric tree,
thus it is $\CAT(k)$ for any $k \in \R$,
whereas the local extendability of geodesics fails.
We put $y_i:=(t_i, t_i)$ and $E:=\{y_i\}_{i = 1}^{\infty}$. 
For $i > j$, we have
\[d_{\BB}(y_i, y_j) = 2t_j. \]
Thus for $i > j > k$ we have 
\[ 2t_k = d_{\BB}(y_i, y_k) = d_{\BB}(y_j, y_k) \ge d_{\BB}(y_i, y_j) = 2t_j. \]
Therefore the infinite set $E$ satisfies the $\SRA(\aa)$-condition for every $\aa > 0$
(moreover, it is a ultrametric space).
Compare this with Proposition~\ref{CATExt-are-ATB}.

\begin{example}
Consider $\{t_i\}_{i=1}^{\infty}$ given by $t_i = 2^{-i+1}$.
In this case $(\BB,d_{\BB})$ satisfies the doubling condition.
Note also that the inclusion $\id:\BB \longrightarrow \R^2$ provides
a bi-Lipschitz embedding into the Euclidean plane
(then what is missed from Lemma~\ref{LiftLemma}(\ref{Adv}) is the last condition (\ref{StarGG})). 
\end{example}

\begin{example}
Take $\{t_i\}_{i=1}^{\infty}$ given by $t_i = i^{-1}$.
Then $E$ provides a (discrete) bounded self-contracted curve of the infinite length. 
\end{example}

\subsection{Laakso graphs}

 \begin{figure}
   \includegraphics[scale=0.25]{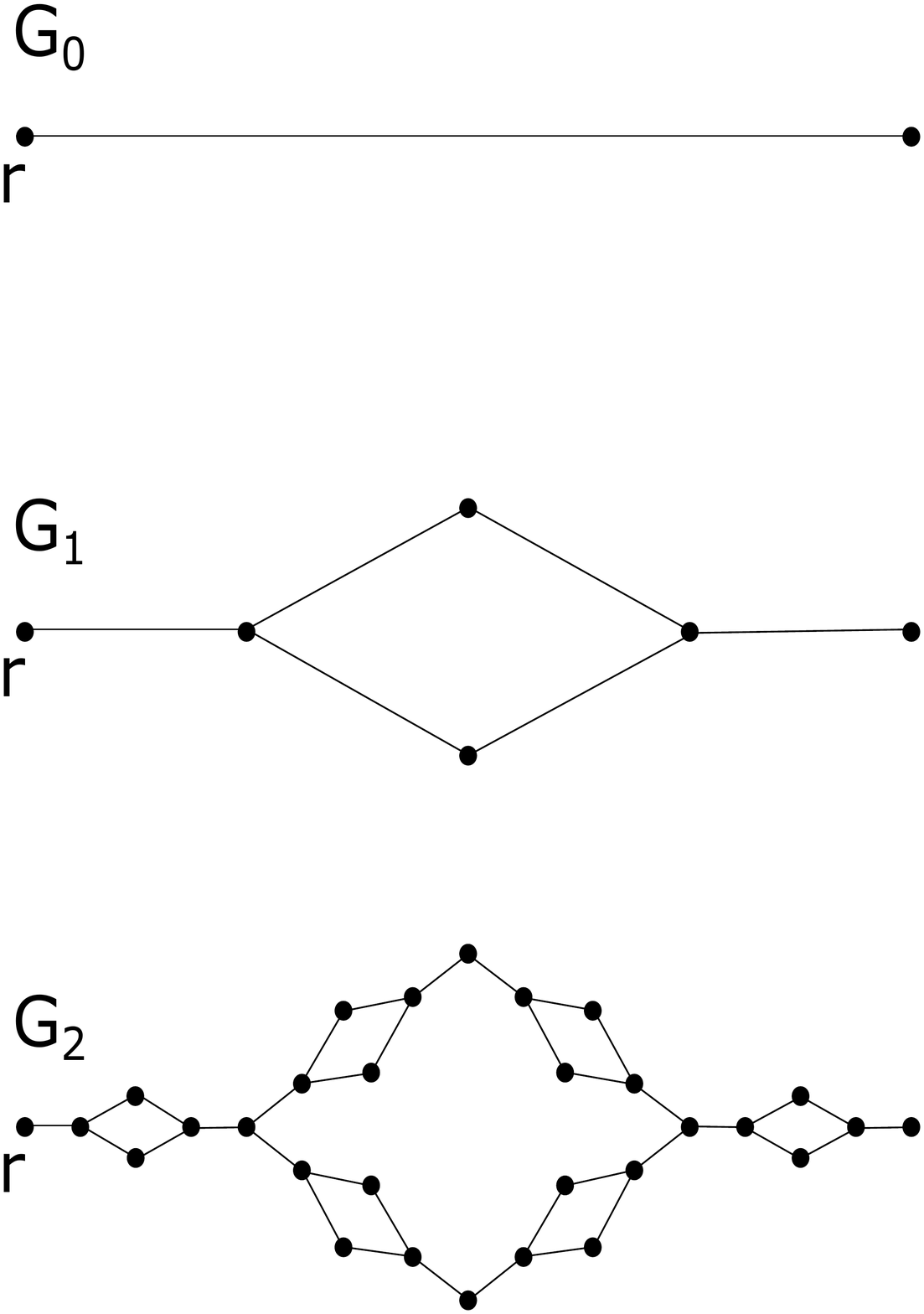}
   \caption{Laakso graphs.}
  \label{Laakso-gr}
 \end{figure}

Next we are going to find $\SRA(\aa)$-subsets in \emph{Laakso graphs}.
Let us give the definition and basic properties of Laakso graphs following \cite{MN-MConv}
(see also \cite{Laakso, LP-BilEmb}).
The family of Laakso graphs $\{G_N\}_{N=0}^{\infty}$ is defined via induction
as in Figure~\ref{Laakso-gr}. 
Precisely, we start with $G_0$ consisting of one edge of unit length,
and $G_N$ is constructed by six copies of $G_{N-1}$ whose metrics are scaled
by a factor of $1/4$.
Glue four of these copies cyclically by identifying pairs of endpoints,
and attach at two opposite gluing points the remaining two copies.
Then $G_N$ is the resulting graph equipped with the intrinsic distance
(thus the diameter of $G_N$ is again $1$).
Laakso graphs satisfy the doubling condition and local Poincar\'e inequality.

\begin{proposition}\label{LaaksoProp}
For any $n,N \in \N$ with $n \le N$,
there exists an $n$-point $\SRA(3/5)$-subset $X$ in $G_N$.
\end{proposition}

The construction of $X$ will be given by Example~\ref{LaaksoEx} below. 
Before that, we explain the necessary notions.
Fix one of the two vertexes of $G_N$ of degree one.
We call it the \emph{root} of $G_N$ and denote by $r$ (see Figure~\ref{Laakso-gr}).
Note that in no edge the two endpoints are at the same distance from the root.
Thus we can direct all the vertexes from the endpoint which is closer to $r$
to the endpoint further from $r$.
It is easy to see that any vertex has at most two outgoing edges
with respect to this direction structure.
For a vertex which has two outgoing edges,
we say that one of them goes to the \emph{right} and the other one goes to the \emph{left}.
We make such a choice in an arbitrary way for each vertex.

\begin{example}\label{LaaksoEx}
Fix $n, N \in \N$ such that $n \le N$. 
Let $\gg:[0,1] \longrightarrow G_N$ be the unit speed minimizing geodesic
from $r$ to the other end of $G_N$, never going to the right in the sense above.
For each $1 \le i \le n$, we put
\[ y_i :=\gg\bigg( \frac{1}{4}+\frac{1}{16}+ \dots +\frac{1}{4^i} \bigg). \]
Then let $x_i \in G_N$ be the endpoint of the oriented path 
which starts from $y_i$, never goes to the left and has length $4^{-i}$.
We define $X := \{x_1,\dots,x_n\}$.
Note that every $x_i$ is a vertex of $G_N$ by $n \le N$.
\end{example}

\begin{proof}[Proof of Proposition~$\ref{LaaksoProp}$]
We shall show that $X$ in Example~\ref{LaaksoEx} satisfies the $\SRA(3/5)$-condition.
For $1 \le i < k \le n$, we observe by construction 
\begin{align*}
d(x_i, x_k) &= d(x_i, y_i) + d(y_i, y_k) + d(y_k, x_k) \\
&= \frac{1}{4^i} + \bigg( \frac{1}{4^{i + 1}} + \dots + \frac{1}{4^k} \bigg) + \frac{1}{4^k}.
\end{align*}
Hence we have, for $i < j < k$, 
\[ d(x_j, x_k) < d(x_i,x_k) < d(x_i, x_j). \]
Thus, to show the $\SRA(3/5)$-condition, it is sufficient to see
\[ d(x_i,x_j) \le d(x_i, x_k) + \frac{3}{5} d(x_j, x_k). \]
This is done by 
\[ d(x_i,x_j) - d(x_i, x_k)
 = \frac{1}{4^j} -\bigg( \frac{1}{4^{j+1}} + \dots + \frac{1}{4^k} \bigg) - \frac{1}{4^k}
 < \frac{1}{4^j} -\frac{1}{4^{j+1}} =\frac{3}{4^{j+1}}, \]
and
\[ d(x_j, x_k) \ge \frac{1}{4^j} +\frac{1}{4^{j+1}} =\frac{5}{4^{j+1}}. \]
\end{proof}

\section{Questions}\label{Questions}

Recall from Proposition~\ref{GC-is-QC} that
gradient curves of lower semi-continuous, quasi-convex, $\ll$-convex functions in 
locally compact $\RCD(K,\infty)$-spaces are self-contracted.
However, it is not known whether bounded self-contracted curves are rectifiable in $\RCD(K,\infty)$-spaces.
By Theorem~\ref{ATB-to-SC}, to provide an affirmative answer to this question,
it suffices to show that $\RCD(K,\infty)$-spaces are locally $\ATB(\ee)$ for some $0 < \ee < \pi/6$.
It may be in fact more reasonable to consider $\RCD(K,N)$-spaces with $N \in (1,\infty)$,
since the finite-dimensionality is a natural condition for the rectifiability.
Thus we have the following question.   

\begin{Question}
Do $\RCD(K,N)$-spaces satisfy the local $\ATB(\ee)$-condition?
\end{Question}

Notice that the weaker notion of measure contraction property does not imply
$\ATB(\ee)$ as is evident in the Heisenberg group (recall Subsection~\ref{Heisenberg}).



Related to Proposition~\ref{CBBs-are-ATB}, one can also ask the following question.
See \cite[7.19]{BGP} for the notion of the boundary of an Alexandrov space of curvature $\ge 0$.

\begin{Question}
Do the boundaries of finite-dimensional Alexandrov spaces of curvature $\ge 0$ satisfy the local $\ATB(\ee)$-condition?
\end{Question}






\subsection*{Acknowledgements}
NL was partially supported by RFBR grant 17-01-00128.
SO was supported in part by JSPS Grant-in-Aid for Scientific Research (KAKENHI) 15K04844.
Theorem~\ref{MeetDoubling},
Propositions~\ref{CATExt-are-ATB}, \ref{pr:norm}, \ref{Fin-are-ATB} and \ref{SubATB} 
are supported by the Russian Science Foundation under grant 16-11-10039.
We would like to thank Sergei Ivanov for useful discussions.

\bibliography{circle}

\bibliographystyle{plain}


\end{document}